\documentclass[12pt]{amsart}

\title{Kissing numbers of closed hyperbolic manifolds}

\author{Maxime Fortier Bourque}
\address{School of Mathematics and Statistics, University of Glasgow, University Place, Glasgow, United Kingdom, G12 8QQ}
\email{maxime.fortier-bourque@glasgow.ac.uk}

\author{Bram Petri}
\address{Mathematisches Institut, Unversit\"at Bonn, Endenicher Allee 60, 53115 Bonn, Germany}
\email{bpetri@math.uni-bonn.de}

\date{\today}

\usepackage{amsmath,amsfonts,amssymb,amsthm,graphicx,overpic,hyperref}
\usepackage{float,enumitem}
\usepackage{color,xcolor}

\usepackage[margin=2.7cm]{geometry} 
\numberwithin{equation}{section}

\newtheorem{thm}{Theorem}[section]
\newtheorem{prop}[thm]{Proposition}

\newtheorem{cor}[thm]{Corollary}
\newtheorem{lem}[thm]{Lemma}

\theoremstyle{definition}

\newtheorem{rem}[thm]{Remark}

\newcommand{\thmref}[1]{Theorem~\ref{#1}}

\newcommand{\propref}[1]{Proposition~\ref{#1}}
\newcommand{\secref}[1]{Section~\ref{#1}}

\newcommand{\lemref}[1]{Lemma~\ref{#1}}
\newcommand{\corref}[1]{Corollary~\ref{#1}}

\newcommand{\eqnref}[1]{Equation~\eqref{#1}}

\newcommand{\nc}{\newcommand}
\nc{\dmo}{\DeclareMathOperator}
\usepackage{dsfont}

\nc{\abs}[1]{\left| #1 \right|}
\nc{\bigO}[1]{O\left(#1\right)}
\nc{\card}[1]{\left|#1\right|}
\nc{\ceil}[1]{\left\lceil #1 \right\rceil}
\nc{\CC}{\mathbb{C}}
\nc{\dilog}{\mathcal{L}}
\nc{\floor}[1]{\left\lfloor #1 \right\rfloor}
\nc{\ind}{\mathds{1}}
\nc{\ZZ}{\mathbb{Z}}
\nc{\len}[1]{\left| #1 \right|}
\nc{\littleo}[1]{o\left(#1\right)}
\dmo{\Mat}{Mat}
\nc{\NN}{\mathbb{N}}
\nc{\norm}[1]{\left|\left| #1 \right|\right|}
\nc{\QQ}{\mathbb{Q}}
\nc{\RR}{\mathbb{R}}
\nc{\st}[2]{\left\{\, #1 \,:\, #2\,\right\}}
\dmo{\supp}{supp}
\nc{\tr}[1]{\mathrm{tr}\left(#1\right)}
\nc{\what}{\widehat}
\dmo{\im}{Im}
\nc{\eps}{\varepsilon}
\dmo{\li}{li}

\dmo{\area}{area}
\dmo{\conv}{conv}
\dmo{\diam}{diam}
\dmo{\DD}{\mathbb{D}}
\dmo{\dist}{\mathrm{d}}
\nc{\HH}{\mathbb{H}}
\dmo{\Isom}{Isom}
\dmo{\MCG}{MCG}
\dmo{\MPL}{MPL}
\dmo{\Mod}{\mathcal{M}}
\dmo{\PL}{PL}
\nc{\Sphere}{\mathbb{S}}
\dmo{\sys}{sys}
\dmo{\kiss}{Kiss}
\dmo{\Teich}{\mathcal{T}}
\nc{\Torus}{\mathbb{T}}
\dmo{\vol}{vol}
\dmo{\WP}{WP}

\dmo{\convTV}{\;\stackrel{\mathrm{TV}}{\longrightarrow}\;}
\nc{\ExV}[2]{\mathbb{E}_{#1}\left[#2\right]}
\dmo{\EE}{\mathbb{E}}
\nc{\Pro}[2]{\mathbb{P}_{#1}\left[#2\right]}
\dmo{\PP}{\mathbb{P}}
\nc{\distTV}[2]{\mathrm{d}_{\rm TV}\left(#1,#2\right)}
\dmo{\UU}{\mathbb{U}}
\nc{\Var}[2]{\mathbb{V}\mathrm{ar}_{#1}\left[#2\right]}

\dmo{\alt}{\mathfrak{A}}
\dmo{\Aut}{Aut}
\dmo{\Fix}{Fix}
\dmo{\GL}{GL}
\dmo{\Hom}{Hom}
\dmo{\id}{Id}
\dmo{\PSL}{PSL}
\dmo{\PO}{PO}
\dmo{\Rep}{Rep}
\dmo{\SL}{SL}
\dmo{\SO}{SO}
\dmo{\sym}{\mathfrak{S}}
\dmo{\inv}{\mathcal{I}}
\dmo{\orb}{\mathcal{O}}
\dmo{\stab}{Stab}

\dmo{\calA}{\mathcal{A}}
\dmo{\calB}{\mathcal{B}}
\dmo{\calC}{\mathcal{C}}
\dmo{\calD}{\mathcal{D}}
\dmo{\calE}{\mathcal{E}}
\dmo{\calF}{\mathcal{F}}
\dmo{\calG}{\mathcal{G}}
\dmo{\calH}{\mathcal{H}}
\dmo{\calI}{\mathcal{I}}
\dmo{\calJ}{\mathcal{J}}
\dmo{\calK}{\mathcal{K}}
\dmo{\calL}{\mathcal{L}}
\dmo{\calM}{\mathcal{M}}
\dmo{\calN}{\mathcal{N}}
\dmo{\calO}{\mathcal{O}}
\dmo{\calP}{\mathcal{P}}
\dmo{\calQ}{\mathcal{Q}}
\dmo{\calR}{\mathcal{R}}
\dmo{\calS}{\mathcal{S}}
\dmo{\calT}{\mathcal{T}}
\dmo{\calU}{\mathcal{U}}
\dmo{\calV}{\mathcal{V}}
\dmo{\calW}{\mathcal{W}}
\dmo{\calX}{\mathcal{X}}
\dmo{\calY}{\mathcal{Y}}
\dmo{\calZ}{\mathcal{Z}}

\begin{document}

\begin{abstract}
We prove an upper bound for the number of shortest closed geodesics in a closed hyperbolic manifold of any dimension in terms of its volume and systole, genera\-lizing a theorem of Parlier for surfaces. We also obtain bounds on the number of primitive closed geodesics with length in a given interval that are uniform for all closed hyperbolic manifolds with bounded geometry. The proofs rely on the Selberg trace formula.
\end{abstract}

\maketitle

\section{Introduction}

The \emph{kissing number} $\kiss(M)$ of a Riemannian manifold $M$ is the number of distinct free homotopy classes of non-trivial, oriented, closed geodesics in $M$ that realize its \emph{systole} --- the minimal length among all such geodesics. The question of how large this number can be has been studied by several authors for flat tori and hyperbolic surfaces.

\subsection*{Flat tori} If $M$ is an $n$-dimensional flat torus, then it is isometric to ${\RR}^n / \Lambda$ for some lattice $\Lambda \subset \RR^n$ and its kissing number is perhaps a more familiar quantity, obtained as follows. Start growing spheres of equal radius at all the points in $\Lambda$ until two of them become tangent. Then $\kiss(M)$ is equal to the number of spheres tangent to (or \emph{kissing}) any given sphere in the resulting packing. This is a much studied quantity (see \cite{PfenderZiegler}), yet lattices with ma\-xi\-mal kissing number are only known in dimensions $1$ to $9$ and $24$ \cite[p.22]{ConwaySloane}. The largest kissing number among lattices in $\RR^n$ was recently shown to grow exponentially in $n$ \cite{Vladut} (the upper bound was proved in \cite{KabLev}).

\subsection*{Hyperbolic surfaces}

 Among all complete hyperbolic metrics of finite area on an o\-rien\-table surface of genus $g$ with $p$ punctures, the metrics that maximize the kissing number in their respective moduli spaces are only known for $(g,p) =(0,4)$, $(1,1)$, $(1,2)$ and $(2,0)$ \cite{SchmutzKiss}. In large genus, the best known examples have kissing number growing faster than $g^{\frac{4}{3} - \eps}$ for every $\eps>0$ \cite{Schmutz}. Furthermore, the kissing number of hyperbolic surfaces of signature $(g,p)$ is bounded above by a sub-quadratic function of $g+p$  \cite{Par,FanPar}. 

\subsection*{Hyperbolic manifolds} 
Our main result bounds the kissing number of a closed hyperbolic manifold $M$ in terms of its volume $\vol(M)$ and systole $\sys(M)$, generalizing Parlier's inequality \cite{Par} to all dimensions.

\begin{thm} \label{thm:main}
For every $n\geq 2$, there exists a constant $A_{n}>0$ such that 
\begin{equation*}
\kiss(M) \leq A_{n} \vol(M) \,\frac{e^{(n-1) \sys(M)/2}}{\sys(M)}
\end{equation*}
for every closed hyperbolic $n$-manifold $M$.
\end{thm}

For manifolds with small systole, a stronger inequality of the form
\begin{equation*}
\kiss(M) \leq A_n' \vol(M) \sys(M)^{\lfloor\frac{n-2}{2}\rfloor / \lfloor\frac{n+1}{2}\rfloor} 
\end{equation*}
follows from estimates on the volume of Margulis tubes around short geodesics due to Keen \cite{Keen} in dimension $2$ and Buser \cite{BuserDim3} in higher dimensions. As such, our contribution is really to the case of manifolds whose systole is uniformly bounded from below. 

Combining \thmref{thm:main} with a standard volume bound for the systole of closed hyperbolic manifolds yields the following simpler inequality.

\begin{cor}  \label{cor:volume}
For every $n\geq 2$, there exists a constant $A_n''>0$ such that 
\begin{equation*} 
\kiss(M) \leq A_n'' \, \frac{\vol(M)^2}{\log(1+\vol(M))}
\end{equation*}
for every closed hyperbolic $n$-manifold $M$.
\end{cor} 

In dimension 2, we recover Parlier's bounds
\[
\kiss(M) \leq U\, \frac{e^{\sys(M) / 2}}{\sys(M)} g \leq  V \, \frac{g^2}{\log g}
\]
with a very different proof and a smaller constant $U\approx 63.71$ (compared to $200$ previously\footnote{In \cite{SchmutzKiss}, \cite{Par} and \cite{FanPar}, the kissing number is defined as the number of shortest unoriented geodesics. We count oriented geodesics instead because that agrees with the usual convention in the Euclidean setting and is well adapted to our proof. We therefore multiplied Parlier's 100 by 2.}), where $g$ is the genus of the closed oriented hyperbolic surface $M$. 

\subsection*{Comments on the proof}

The proof of \thmref{thm:main} relies on the Selberg trace formula, which links the spectrum of the Laplace operator on a hyperbolic manifold to its length spectrum via pairs of functions that are Fourier transforms of one another. The idea of the proof is to look for a function that picks up the bottom part of the length spectrum and whose Fourier transform does not take negative values on the Laplace spectrum. This strategy was inspired by a similar approach for bounding the density of sphere packings \cite{CohnElkies} which was recently used to prove the optimality of the $E_8$ and Leech lattices in dimensions $8$ and $24$ \cite{Viazovska,CKMRV}. In the Euclidean setting, the role of the Selberg trace formula is played by the Poisson summation formula.

\subsection*{Uniform length spectrum bounds}

Given a closed hyperbolic $n$-manifold $M$, we denote the set of primitive, oriented, closed geodesics in $M$ by $\calP(M)$ and the subset whose lengths lie in an interval $[a,b]$ by $\calP_{[a,b]}(M)$. The prime geodesic theorem \cite{Huber,Gangolli,DeGeorge} states that the cardinality of $\calP_{[0,L]}(M)$ is asymptotic to
\begin{equation} \label{eq:asymptotic}
\frac{e^{(n-1) L}}{(n-1) L} \quad \text{ as }L \to \infty.
\end{equation}

This is a remarkable fact, in part because the ultimate behavior does not depend on anything except the dimension of the manifold. On the other hand, it gives no information about what happens if we vary not only the length, but also the underlying manifold.

As a further application of our methods, we obtain uniform upper and lower bounds for the number of primitive geodesics whose lengths fall in a short interval that apply to all manifolds with systole bounded below. 

\begin{thm}\label{thm:unifpgt}
For every $n\geq 2$ and $\delta>0$, there exist constants $B_{n,\delta}, C_{n,\delta}, D_{n,\delta} >0$ such that for every closed hyperbolic $n$-manifold $M$ with $\sys(M)\geq  2\delta$, and every $L>0$, we have
 \begin{equation*} \label{eq:interval}
C_{n,\delta} \frac{e^{(n-1)L}}{L} - D_{n,\delta} \vol(M) \frac{e^{\frac{(n-1)}{2} L}}{L} \leq \#\calP_{[L-\delta,L+\delta]}(M) \leq  B_{n,\delta} \vol(M) \frac{e^{(n-1)L}}{L}.
\end{equation*}
\end{thm}

One can think of the lower bound as an analogue of Bertrand's postulate in number theory, for it implies that for all large enough $N>0$ there is a primitive closed geodesic $\gamma$ in $M$ whose norm $e^{\ell(\gamma)}$ is between $N$ and $e^{2\delta}N$. 

By integrating the above inequalities, we obtain similar bounds for the number of primitive geodesics of length at most $L$, matching the asymptotic \eqref{eq:asymptotic} up to multiplicative constants.

\begin{cor}\label{cor:unifpgt}
For every $n\geq 2$ and $\delta>0$, there exist constants $B_{n,\delta}',C_{n,\delta}',D_{n,\delta}'>0$ such that for every closed hyperbolic $n$-manifold $M$ with $\sys(M)\geq 2\delta$, and every $L>0$,
\[  C_{n,\delta}' \frac{e^{(n-1)L}}{L} - D_{n,\delta}' \vol(M) \frac{e^{\frac{(n-1)}{2} L}}{L} \leq \#\calP_{[0,L]}(M) \leq  B_{n,\delta}' \vol(M) \frac{e^{(n-1)L}}{L}.\]
\end{cor}

Note that the lower bound reproves the well-known inequality
\[
\sys(M) \leq \frac{2}{n-1} \log(\vol(M)) + \text{const.}
\] 
albeit in a somewhat complicated way.

In dimension $n\geq 3$, the existence of a constant $E(n,\delta, v)>0$ such that all closed hyperbolic $n$-manifolds $M$ with systole at least $2 \delta$ and volume at most $v$ satisfy
\[
  \#\calP_{[0,L]}(M)\leq  E(n,\delta,v) \frac{e^{(n-1)L}}{L} \quad \quad \text{for all }L>0
\]
and 
\[
\#\calP_{[0,L]}(M)\geq  \frac{1}{E(n,\delta,v)} \frac{e^{(n-1)L}}{L} \quad \quad \text{for all }L\text{ large enough}
\]
can be deduced from the prime geodesic theorem and the fact that there are only finitely many hyperbolic $n$-manifolds with bounded geometry (see \cite{Wang} for $n\geq 4$ and \cite[Theorem E.4.8]{BP} for $n=3$). The advantage of our results is that they make the dependence on volume explicit.

In dimension $2$, there are infinitely many manifolds of a given volume with systole bounded below, so the fact that length spectrum bounds hold uniformly for all of them is not obvious. A uniform upper bound without the requirement that the systole be bounded below but with faster growth rate $B\vol(M) e^L$ was previously obtained in \cite[p.162]{Buser}. In \cite[Lemma 5.1]{ABG}, surfaces $M$ with $\sys(M) \approx e^{-L/4}$ and  $\#\calP_{[0,L]} \geq C \vol(M) e^{L/4}$ are constructed. Our uniform lower bound for thick surfaces appears to be new, and grows faster as a function of $L$ for a fixed genus.

Finally, we note that even though we will not pursue this (except in dimension $2$), all the constants above are effectively computable.

\section{Short geodesics} \label{sec:thin}

We first comment on the number of short geodesics in closed hyperbolic manifolds. In dimension $2$, the collar lemma \cite{Keen} implies that in a closed oriented hyperbolic surface of genus $g$, distinct primitive (unoriented) closed geodesics of length at most $2 \sinh^{-1}(1)$ are disjoint, so there are at most $3g-3$ of them. Since the area of a closed oriented hyperbolic surface $M$ of genus $g$ is $4\pi(g-1)$, this implies that
\begin{equation} \label{eq:thin2}
\kiss(M) \leq \#\calP_{[0,2 \sinh^{-1}(1)]}(M) \leq \frac{3}{2 \pi} \vol(M)
\end{equation}
whenever $\sys(M) \leq 2 \sinh^{-1}(1)$.

For closed oriented  hyperbolic manifolds of dimension $n \geq 3$, Buser \cite[\S 4]{BuserDim3} proved that any primitive closed geodesic $\gamma$ of length $\ell(\gamma)\leq 4^{-(n+2)}$ has a tubular neighborhood $T_\gamma$ that satsifies
\[
\vol(T_\gamma) \geq K_n \, \ell(\gamma)^{-\lfloor\frac{n-2}{2}\rfloor / \lfloor\frac{n+1}{2}\rfloor}
\]
for come constant $K_ n>0$ depending on dimension only. As Buser notes, in dimension $3$ the lower bound is constant, which is the best one can hope for in view of Thurston's Dehn filling theorem. 

Since tubes $T_\gamma$ corresponding to primitive (unoriented) closed geodesics of length at most $4^{-(n+2)}$ are pairwise disjoint \cite[Theorem 4.13]{BuserDim3}, we obtain 
\begin{equation} \label{eq:thin3}
\kiss(M) \leq \#\calP_{[0,L]}(M) \leq A_n' \vol(M)  L^{\lfloor\frac{n-2}{2}\rfloor / \lfloor\frac{n+1}{2}\rfloor}
\end{equation}
whenever $\sys(M) \leq L \leq 4^{-(n+2)}$, where $A_n':= 2/K_n$.

If $M$ is non-orientable, we can pass to the orientable double cover $\widetilde M$ which satisfies $\#\calP_{[0,L]}(M) \leq \#\calP_{[0,2L]}(\widetilde M)$ and $\vol(\widetilde M) = 2 \vol(M)$ to get similar inequalities with additional factors of $2$.

As claimed in the introduction, inequalities \eqref{eq:thin2} and \eqref{eq:thin3} imply \thmref{thm:main} for ma\-ni\-folds with small systole. This is because $e^{(n-1)x/2} / x \geq (n-1)e/2$ for every $x>0$, so the term involving the systole in \thmref{thm:main} is larger than a constant times $\sys(M)^{\lfloor\frac{n-2}{2}\rfloor / \lfloor\frac{n+1}{2}\rfloor}$ provided we restrict $\sys(M)$ to some interval $(0,\eps_n]$.

\section{The Selberg trace formula}

Selberg introduced his trace formula for discrete groups of isometries of the hyperbolic plane in \cite{Selberg} (see \cite{Buser} for a modern exposition). This was generalized to closed hyperbolic manifolds of any dimension in \cite{Randol}, \cite{Deitmar} and \cite{Parnovskii}. We will follow the notation from \cite{Randol} and \cite{Buser}.

Let $M$ be a closed hyperbolic manifold of dimension $n\geq 2$, that is, a quotient of the hyperbolic space $\HH^n$ by a discrete, torsion-free, cocompact group $\Gamma$ of isometries (not necessarily orientation-preserving). Let 
\[
0=\lambda_0 < \lambda_1 \leq \lambda_2 \leq \cdots
\]
be the eigenvalues of the (negative) Laplacian on $M$, repeated according to their multiplicity. For each integer $j\geq 0$, let
\[
r_j : = 
\begin{cases}
i \,\sqrt{\frac{(n-1)^2}{4} - \lambda_j} & \text{if } 0 \leq \lambda_j \leq \frac{(n-1)^2}{4} \\
\sqrt{\lambda_j - \frac{(n-1)^2}{4}} & \text{if } \lambda_j > \frac{(n-1)^2}{4}
\end{cases}
\]
where the non-negative square root is used and $i = \sqrt{-1}$ is the imaginary unit.

Let $\calC(M)$ be the set of closed oriented geodesics in $M$. These are in one-to-one corres\-pondance with non-trivial conjugacy classes in the group $\Gamma$. The length of a geodesic $\gamma \in \calC(M)$ is denoted by $\ell(\gamma)$ and its norm is defined as $N_\gamma:= e^{\ell(\gamma)}$.  A geodesic is called \emph{primitive} if it is not a proper power of another geodesic. For $\gamma \in \calC(M)$, we set $\Lambda(\gamma):=\ell(\gamma_0)$ where $\gamma_0$ is the unique primitive closed geodesic such that $\gamma=\gamma_0^m$ for some $m\geq 1$. Finally, we let $P_\gamma$ be the holonomy around $\gamma$, restricted to its normal bundle. That is, given a point $p=\gamma(t)$ and a tangent vector $v \in T_p M$ orthogonal to $\gamma'(t)$, the vector $P_\gamma(v)$ is obtained by parallel transporting $v$ around $\gamma$. In other words, $P_\gamma$ is the rotational component of the loxodromic isometry corresponding to $\gamma$ in $\Gamma$. We define
\[
D(\gamma):= |\det(I -N_\gamma^{-1} P_\gamma^{-1})|
\]
where $I$ is the identity on  $\gamma'(t)^\perp$. This quantity does not depend on the point $p$.

A pair of functions $(g,h)$ is called \emph{admissible} (in dimension $n$) or an \emph{admissible transform pair} if $g: \RR \to \CC$ is even and integrable, and its \emph{Fourier transform}
\[
h(\xi):= \what{g}(\xi) = \int_{-\infty}^\infty g(x) e^{-i \xi x} \,dx
\]
is holomorphic and satisfies the decay condition
\begin{equation} \label{eq:decay}
|h(\xi)| = O(|\xi|^{-n-\eps})
\end{equation}
in the strip $\st{\xi \in \CC}{|\im \xi|< \frac{n-1}{2}+\eps}$ for some $\eps>0$. This convention for the Fourier transform is sometimes called non-unitary with angular frequency.

The Plancherel density for $\HH^n$ is given by
\[
\Phi_n(r) = \frac{r \tanh(\pi r)}{(2\pi)^\frac{n}{2}(n-2)!!} \prod_{k=0}^{\frac{n-4}{2}}\left(r^2 + \left(k+\frac12\right)^2 \right) \quad \quad \text{if }n\text{ is even}
\]
or
\[
\Phi_n(r) =  \frac{1}{2^\frac{n-1}{2}\pi^\frac{n+1}{2}(n-2)!!} \prod_{k=0}^{\frac{n-3}{2}}\left(r^2 + k^2 \right)  \quad \quad \text{if }n\text{ is odd,}
\]
where $j!!$ is the product of the integers between $1$ and $j$ with the same parity as $j$, and an empty product is equal to $1$. Randol obtains different expressions for $\Phi_n$, but observes that it can be formulated in terms of the classical gamma function using Harish-Chandra's Plancherel formula for spherical functions\footnote{The penultimate equation on p.292 of \cite{Randol} appears to contain a typographical error as it does not coincide with the formula for $\Phi_3$ given a few lines above it.}. The above equations are taken from \cite{Parnovskii}.

The following relationship between the Laplace spectrum and the length spectrum holds for all closed hyperbolic manifolds \cite[p.292]{Randol}.

\begin{thm}[Selberg's trace formula]
Let $M$ be a closed hyperbolic manifold of dimension $n\geq 2$. For any admissible transform pair $(g,h)$ we have
\begin{equation} \label{eq:STF}
\sum_{j=0}^\infty h(r_j) =  \vol(M) \int_0^\infty h(r) \Phi_n(r)\, dr + \sum_{\gamma \in \calC(M)}  \frac{\Lambda(\gamma)}{N_\gamma^{(n-1)/2} D(\gamma)} g(\ell(\gamma)).
\end{equation}
\end{thm}

\begin{rem}
Randol takes two roots $r_j$ for each eigenvalue $\lambda_j$, so his formula differs from the above by a factor of $2$. \eqnref{eq:STF} agrees with \cite[p.253]{Buser} for $n=2$.
\end{rem}

In applications, it will be convenient to estimate the factor $D(\gamma)$ in the trace formula in terms of length alone. This is done in the following lemma.

\begin{lem} \label{lem:det}
For any closed geodesic $\gamma$ in a closed hyperbolic $n$-manifold, we have
\[
(1-N_\gamma^{-1})^{n-1} \leq D(\gamma)\leq (1+N_\gamma^{-1})^{n-1} < 2^{n-1}.
\]
\end{lem}
\begin{proof}
Recall that $D(\gamma) = |\det(I -N_\gamma^{-1} P_\gamma^{-1})|$ is the product of the absolute values of the eigenvalues of $I -N_\gamma^{-1} P_\gamma^{-1}$. To prove the first two inequalities, it suffices to show that all the eigenvalues are between $1-N_\gamma^{-1}$ and $1+N_\gamma^{-1}$ in absolute value. For any vector $v \in T_pM$ we have 
\[
\norm{(I - N_\gamma^{-1} P_\gamma^{-1})v} = \norm{v - N_\gamma^{-1} P_\gamma^{-1}v} \leq \norm{v} + N_\gamma^{-1}\norm{P_\gamma^{-1}v} = (1+N_\gamma^{-1})\norm{v}
\]
by the triangle inequality and the fact that $P_\gamma$ is an isometry. Similarly, \[\norm{(I - N_\gamma^{-1} P_\gamma^{-1})v}\geq  (1-N_\gamma^{-1})\norm{v}.\] Applying these inequalities to the eigenvectors of $I -N_\gamma^{-1} P_\gamma^{-1}$ implies the required bounds on eigenvalues. The last inequality follows from the fact that $N_\gamma = e^{\ell(\gamma)}>1$.
\end{proof}

\section{Kissing numbers}

In this section, we use the Selberg trace formula to obtain a general bound on the kissing numbers of closed hyperbolic manifolds, proving Theorem \ref{thm:main} and \corref{cor:volume}. Even though this method works for all manifolds, the constant we obtain blows up as the systole tends to zero. In that case, we rely on the results from Section \ref{sec:thin} instead.

The idea of the proof is to find a transform pair $(g,h)$ which is well suited for counting shortest closed geodesics. By restricting the signs of these functions, we can obtain bounds for the kissing number.

\begin{prop} \label{prop:strategy}
Let $n\geq 2$ and let $M$ be a closed hyperbolic $n$-manifold. Suppose that $(g,h)$ is an admissible transform pair such that
\begin{itemize}
\item $g(x) \leq 0$ for every $x \geq \sys(M)$;
\item $h(\xi) \geq 0$ for every $ \xi \in\RR \cup i \left[\frac{1-n}{2}, \frac{n-1}{2}\right]$.
\end{itemize}
Then
\begin{equation} \label{eq:strategy}
\kiss(M) \frac{2^{1-n} \sys(M)}{e^{(n-1)\sys(M)/2}} \, |g(\sys(M))|  \leq \vol(M) \int_0^\infty h(r) \Phi_n(r)\, dr.
\end{equation}
\end{prop}

\begin{proof}
The hypothesis implies that $h(r_j) \geq 0$ for every $j\geq 0$. From the Selberg trace formula \eqref{eq:STF}, we get
\[
0 \leq \sum_{j=0}^\infty h(r_j) = \vol(M) \int_0^\infty h(r) \Phi_n(r)\, dr + \sum_{\gamma\in \calC(M)} \frac{\Lambda(\gamma)}{N_\gamma^{(n-1)/2}D(\gamma)}  g(\ell(\gamma))
\]
or
\begin{equation} \label{eq:basic}
\sum_{\gamma\in \calC(M)} \frac{\Lambda(\gamma)}{N_\gamma^{(n-1)/2} D(\gamma)}  |g(\ell(\gamma))| \leq \vol(M) \int_0^\infty h(r) \Phi_n(r)\, dr
\end{equation}
after subtracting the sum from both sides (recall that $g(\ell(\gamma)) \leq 0$ by hypothesis).

If $\gamma$ is a shortest closed geodesic in $M$, then it is primitive so that \[\Lambda(\gamma) = \ell(\gamma) = \sys(M)\] and the corresponding summand in the Selberg trace formula satisfies
\begin{equation*} 
\frac{\Lambda(\gamma)}{N_\gamma^{(n-1)/2} D(\gamma)}  |g(\ell(\gamma))|  \geq  \frac{2^{1-n}\sys(M)}{e^{(n-1)\sys(M)/2}} \, |g(\sys(M))| 
\end{equation*}
according to \lemref{lem:det}. By summing over all the shortest closed geodesics in $M$ and disregarding the other terms in \eqref{eq:basic}, we obtain inequality \eqref{eq:strategy}.
\end{proof}


In order to make the estimate from \propref{prop:strategy} useful, we need to find admissible transform pairs $(g,h)$ satisfying the hypotheses such that $|g(\sys(M))|$ is not too small and the integral of $h$ is not too large. In the proof of the following proposition, we give a simple recipe for obtaining such pairs from a bump function whose Fourier transform satisfies a sign condition. 

\begin{prop} \label{prop:suited}
Let $n\geq 2$ and $\eps>0$, and let $M$ be a closed hyperbolic $n$-manifold with $\sys(M) \geq \eps$. Suppose that $(g_\eps,h_\eps)$ is an admissible transform pair such that 
\begin{itemize}
\item $g_\eps(x) \geq 0$ for every $x \in \RR$, with equality outside $(-\eps,\eps)$;
\item $h_\eps(\xi) \geq 0$ for every $\xi \in \RR \cup i \left[\frac{1-n}{2}, \frac{n-1}{2}\right]$.
\end{itemize}
Then
\begin{equation} \label{eq:suitable}
\kiss(M) \frac{\sys(M)}{e^{(n-1)\sys(M)/2}} \, g_\eps(0)  \leq 2^{n+1}\left(1+e^{\frac{(n-1)\eps}{2}}\right)\vol(M) \int_0^\infty h_\eps(r) \Phi_n(r)\, dr.
\end{equation}
\end{prop}
\begin{proof}
Let $\nu := (n-1)/2$ and $R := \sys(M)$. We first check that
\[
G(x) :=\left(1+e^{\nu \eps}\right) g_\eps(x)+  e^{\nu \eps}\, \
\frac{g_\eps(x-R+\varepsilon)+g_\eps(x+R-\varepsilon)}{2} -\frac{g_\eps(x-R)+g_\eps(x+R)}{2}
\]
and $H(\xi) := \what{G}(\xi)$ form an admissible transform pair satisfying the hypotheses of \propref{prop:strategy}. 

If $f:\RR\to \CC$ is given by $f(x) = f_0(x-a)$ for some integrable function $f_0$ and some $a\in \RR$, then its Fourier transform satisfies $\widehat{f}(\xi) = e^{i\xi a} \widehat{f_0}(\xi)$. Together with Euler's formula $e^{iz} + e^{-iz} = 2 \cos z$, this implies that
\[
H(\xi) = \left(1+e^{\nu \eps}+e^{\nu \eps}\cos((R-\varepsilon)\xi) - \cos(R\xi)\right)h_\eps(\xi).
\]
We will use this kind of transformation rule without further mention in the sequel.

It is clear that $G$ is even and integrable since the same is true for $g_\eps$. The assumption on the support of $g_\eps$ and the hypothesis $R\geq \varepsilon$ further imply that $G(x) = -g_\eps(x-R)/2 \leq 0$ whenever $x \geq R$.  In particular, $G(R)=-g_\eps(0)/2$.

On the real line we have
\[
0 \leq1+e^{\nu \eps}+e^{\nu \eps}\cos((R-\varepsilon)\xi) - \cos(R\xi) \leq 2\left(1+e^{\nu \eps}\right)  
\]
so that $H(\xi) \geq 0$ and
\[
 \int_0^\infty H(r) \Phi_n(r)\, dr \leq 2\left(1+e^{\nu \eps}\right) \int_0^\infty h_\eps(\xi)  \Phi_n(r)\,dr.
\]
In fact, the factor $1+e^{\nu \eps}+e^{\nu \eps}\cos((R-\varepsilon)\xi) - \cos(R\xi)$ is bounded on any horizontal strip of bounded height, so that $H$ satisfies the decay condition \eqref{eq:decay} in addition to being holomorphic wherever $h_\eps$ is. This shows that $(G,H)$ is an admissible pair.

If $\xi = i t$ for some $t\in \left[0,\nu\right]$, then
\[
e^{\nu \eps}\cos((R-\varepsilon)\xi) = e^{\nu \eps} \cosh((R-\varepsilon) t) \geq \cosh(Rt) = \cos(R \xi).
\]
This implies that $H(\xi) \geq 0$ for every $\xi \in  i \left[-\nu, \nu\right]$. Applying \propref{prop:strategy} to the pair $(G,H)$ yields the desired inequality.
\end{proof}

It only remains to exhibit a pair $(g_\eps,h_\eps)$ satisfying the hypotheses of \propref{prop:suited}, which we do in the following lemma.

\begin{lem} \label{lem:convo}
Let $n \geq 2$ and $\eps>0,$ and let $g_\eps:\RR\to\RR$ be defined by
\[
g_\eps := 
\begin{cases} \left(\chi_{\left[\frac{-\eps}{n+2},\frac{\eps}{n+2}\right]}\right)^{\ast (n+2)} & \text{if }n\text{ is even} \\
 \left(\chi_{\left[\frac{-\eps}{n+1},\frac{\eps}{n+1}\right]}\right)^{\ast (n+1)} & \text{if }n\text{ is odd},
\end{cases}
\]
 where $f^{\ast j}$ denotes the $j$-th convolution of $f$ with itself and $\chi_{A}$ is the cha\-rac\-te\-ris\-tic function of the set $A$. Then $g_\eps$ and its Fourier transform $h_\eps$ satisfy the hypotheses of \propref{prop:suited}.
\end{lem}
\begin{proof}
Recall that the convolution of two integrable functions $\sigma$ and $\tau$ is defined by
\[
(\sigma \ast \tau)(x) := \int_{-\infty}^\infty \sigma(x-y)\tau(y)\,dy
\]
for any $x \in \RR$. It is easy to show that the essential supports of these functions satisfy
\[
\supp(\sigma \ast \tau) \subset \supp(\sigma)+\supp(\tau).
\] 
By induction, it follows that the support of $g_\eps$ is contained in $[-\eps,\eps]$. Moreover, $g_\eps$ is even, integrable and non-negative.

By the convolution theorem (see for instance \cite[Proposition 5.1.11]{SteinShakarchi}), the Fourier transform of $g_\eps$ satisfies
\[
h_\eps = \what{g_\eps}= 
\begin{cases} \left(\what\chi_{\left[\frac{-\eps}{n+2},\frac{\eps}{n+2}\right]}\right)^{n+2} & \text{if }n\text{ is even} \\
\left(\what\chi_{\left[\frac{-\eps}{n+1},\frac{\eps}{n+1}\right]}\right)^{n+1} & \text{if }n\text{ is odd}.
\end{cases}
\]
We chose the exponents in such a way that $h_\eps$ is an even power of a function which is real-valued in $\RR \cup i \RR$, making it non-negative there. Furthermore, $h_\eps$ is entire and satisfies the decay condition \eqref{eq:decay}. Indeed, for any $a>0$ we have
\[
\what\chi_{[-a,a]}(\xi) = \int_{-a}^a e^{-i\xi x}\,dx = \frac{e^{i a \xi} - e^{-i a \xi}}{i\xi} = \frac{2 \sin(a \xi)}{\xi} ,
\]
which has a removable singularity at the origin. Since the sine function is bounded on any horizontal strip of bounded height, we have $|h_\eps(\xi)|= O(|\xi|^{-n-1})$ in the strip \[\st{\xi \in \CC}{|\im \xi|< \frac{n-1}{2}+1}.\]
The pair $(g_\eps,h_\eps)$ is therefore admissible.
\end{proof}

We can now prove that the kissing number is bounded by a function of the volume and the systole.

\begin{proof}[Proof of \thmref{thm:main}] 
For manifolds $M$ with $\sys(M) \leq \eps_n :=4^{-(n+3)}< \sinh^{-1}(1)$, the theorem was proved in Section \ref{sec:thin}. As such, we may assume that $\sys(M) \geq \eps_n$ and apply \propref{prop:suited} to the pair $(g_{\eps_n},h_{\eps_n})$ from \lemref{lem:convo}. The theorem follows by setting
\[
A_n := \frac{2^{n+1}}{g_{\eps_n}(0)}\left(1+e^\frac{(n-1)\eps_n}{2}\right) \int_0^\infty h_{\eps_n}(\xi) \Phi_n(r) \,dr.
\]
\end{proof}

\begin{rem}\label{rem:cst}
For a closed orientable hyperbolic surface $M$, we do not need to rely on Lemma \ref{lem:det} since $N_\gamma^{1/2} D(\gamma)$ simplifies to $2 \sinh(\ell(\gamma)/2)$.
As such, we obtain the better inequality
\[
\frac{\sys(M)}{\sinh(\sys(M)/2) \vol(M)}\cdot \kiss(M) \leq \frac{2(1+e^{\eps/2})}{\pi g_\varepsilon(0)}  \int_0^\infty h_\eps(r) \cdot r \cdot \tanh(\pi r) dr =: C_{2,\varepsilon}',
\]
whenever $\sys(M) \geq \eps$. Evaluating this at $\eps=2\sinh^{-1}(1)$, we get a constant of 
\[C_{2,2\sinh^{-1}(1)}' =\frac{48 (1+e^{\sinh^{-1}(1)})}{\pi (\sinh^{-1}(1))^3} \int_0^\infty \frac{\sin(\sinh^{-1}(1) r/2)^4}{r^3} \tanh(r)dr = 10.1391 \ldots . \]
To obtain the constant $U \approx 63.71$ stated in the introduction, we multiply the above by $4\pi$ to replace area with genus and divide by $2$ to replace $\sinh(\sys(M)/2)$ with $\exp(\sys(M)/2)$.
\end{rem}

Next, we prove the corollary stating that the kissing number is a sub-quadratic function of the volume.

\begin{proof}[Proof of \corref{cor:volume}]
The proof is standard. It uses the fact that the volume of balls grows exponentially with the radius to get a logarithmic upper bound on the systole, which combined with \thmref{thm:main} gives what we want. The precise details are written below.

Let $r_n:= 1 / (n-1)$. This is chosen so that both 
\[
2(e^{(n-1)r} - 1) \geq e^{(n-1)r} \quad \text{whenever } r\geq r_n
\]
and the function $f(r):=e^{(n-1)r/2} / r$ is increasing for $r\geq 2r_n$.

\thmref{thm:main} and the discussion in \secref{sec:thin} imply that there is a constant $b_n>0$ such that $\kiss(M) \leq b_n \vol(M)$ whenever $\sys(M)\leq 2r_n$.  Since there is a lower bound $v_n>0$ for the volume of all closed hyperbolic $n$-manifolds \cite{KazhdanMargulis}, the ratio $\vol(M)/\log(1+\vol(M))$ is also bounded away from zero. Therefore, the inequality
\[
\kiss(M) \leq  b_n \vol(M) \leq c_n \, \frac{\vol(M)^2}{\log(1+\vol(M))}
\]
holds for some constant $c_n>0$ whenever $\sys(M)\leq 2r_n$.

Now assume that $\sys(M) > 2r_n$. The volume of a ball $B_r$ of radius $r>r_n$ in $\HH^n$ satisfies
\[
\vol(B_r) = \frac{2 \pi^{n/2}}{\Gamma(n/2)} \int_0^r [\sinh(x)]^{(n-1)}\,dx \geq \frac{2\pi^{n/2}}{\Gamma(n/2)} \cdot \frac{e^{(n-1)r} - 1}{(n-1)2^{n-1}} \geq \frac{\pi^{n/2}}{\Gamma(n/2)} \cdot \frac{e^{(n-1)r}}{(n-1)2^{n-1}}
\]
where $\Gamma$ is the classical gamma function. Here the first inequality is proved using the racetrack principle and the second inequality follows from the hypothesis $r>r_n$. Since any open ball of radius $\sys(M)/2$ in $M$ is embedded, we find that
\[
\vol(M) \geq \vol(B_{\sys(M) / 2}) > d_n  \, e^{(n-1)\sys(M)/2}
\]
for some constant $d_n>0$, which we may assume is less than the volume bound $v_n$. This is so that $\vol(M)/d_n \geq v_n/d_n > 1$, which implies that there exists a constant $a_n>0$ such that $(1+\vol(M))^{a_n} \leq \vol(M)/d_n$ for all $M$.  Indeed, a direct computation shows that $a_n = \log(v_n/d_n)/\log(1+v_n)$ will do. 

Since $2r_n < \sys(M) < \frac{2}{n-1}\log(\vol(M)/d_n)$ and the function $f$ is increasing in that range, we get an inequality of the form
\[ \frac{e^{(n-1)\sys(M)/2}}{\sys(M)} \leq \frac{\vol(M) / d_n}{\frac{2}{n-1}\log(\vol(M)/d_n)} \leq \frac{\vol(M)/d_n}{\frac{2 a_n}{n-1}\log(1+\vol(M))} . \]
So, filling in Theorem \ref{thm:main}, and combining all the constants into a single one we obtain
\[\kiss(M) \leq A_{n} \vol(M) \,\frac{e^{(n-1) \sys(M)/2}}{\sys(M)} \leq A_n''\frac{\vol(M)^2}{\log(1+\vol(M))}
\]
as required.
\end{proof}

\section{Length spectrum bounds}

In this section, we prove Theorem \ref{thm:unifpgt} and Corollary \ref{cor:unifpgt}. We treat the upper and lower bounds separately and start with the former.

\subsection{Upper bound}

We will prove the following upper bound:

\begin{prop} \label{prop:primupper}
For every $n\geq 2$ and $\delta>0$, there exists a constant $B_{n,\delta}>0$ such that for every closed hyperbolic $n$-manifold $M$ with $\sys(M)\geq  2\delta$, and every $L>0$, 
\[ \# \calP_{[L-\delta,L+\delta]}(M) \leq B_{n,\delta} \vol(M) \frac{e^{(n-1)L}}{L}.\]
\end{prop}
\begin{proof}
For ease of notation, we write $\nu = (n-1)/2$. Let $\eps:=2\delta$ and let $(g_\eps,h_\eps)$ be the admissible transform pair provided by \lemref{lem:convo}. It is easy to show that $g_\eps$ is continuous and has support equal to $[-\eps,\eps]$. Furthermore, $h_\eps$ is non-negative in $\RR \cup i[-\nu,\nu]$ by construction.

Given $L>0$, consider the function defined by
\[G^L_\eps(x) := \frac{g_\eps(x-L)+g_\eps(x+L)}{2} -g_\eps(x)\]
for every $x \in \RR$. Then 
\[H^L_\eps(\xi):=\what{G^L_\eps}(\xi) = (\cos(L\xi)-1)h_\eps(\xi)\]
 is non-positive and bounded below by $-2 h_\eps$ in $\RR$, and the pair $(G^L_\eps,H^L_\eps)$ is admissible. 

Let $M$ be a closed hyperbolic $n$-manifold such that $\sys(M)\geq 2\delta = \eps$. Then $g_\eps(\ell(\gamma)) = 0$ for every closed geodesic $\gamma$ in $M$ so that $G^L_\eps$ is non-negative on the length spectrum of $M$. We also have $2^{1-n} \leq(1-N_\gamma^{-1})/D(\gamma)$ for every closed geodesic $\gamma$ by \lemref{lem:det}. Furthermore, if $\ell(\gamma)$ is in the interval $[L-\delta, L+ \delta]$ then
\[G^L_\eps(\ell(\gamma)) = \frac{g_\eps(\ell(\gamma)-L)+g_\eps(\ell(\gamma)+L)}{2} \geq \frac{g_\eps(\ell(\gamma)-L)}{2} \geq \frac{\mu}{2}\]
where $\mu>0$ is the minimum of $g_\eps$ on the interval $[-\delta,\delta]$.

Recall that $\calP_{[L-\delta,L+\delta]}(M)$ denotes the set of primitive closed oriented geodesics $\gamma$ in $M$ whose length $\ell(\gamma)= \Lambda(\gamma)$ is contained in the interval $[L-\delta, L+\delta]$. The inequalities from the previous paragraph combine to give
\begin{equation} \label{eq:prim1}
\mu 2^{-n} \sum_{\gamma \in\calP_{[L-\delta,L+\delta]}(M)} \frac{\ell(\gamma)}{e^{\nu \ell(\gamma)}} \leq \sum_{\gamma \in \calP(M)}  \frac{\Lambda(\gamma)}{{N_\gamma}^{\nu} D(\gamma)} G^L_\eps(\ell(\gamma)) \leq \sum_{\gamma \in \calC(M)}  \frac{\Lambda(\gamma)}{{N_\gamma}^{\nu} D(\gamma)} G^L_\eps(\ell(\gamma)).
\end{equation}
Observe that the function $x e^{-\nu x}$ has a unique local maximum at $x = 1/\nu$, where it takes the value $1/(\nu e) < 1$. Therefore, its minimum on any compact interval is attained at one of the endpoints. Since $\ell(\gamma) \geq 2\delta$ for every $\gamma \subset M$, we have
\begin{equation} \label{eq:prim2}
 \min\left( \frac{2\delta}{e^{2\nu \delta}}, \frac{L+\delta}{e^{\nu(L+\delta)}} \right)\leq \frac{\ell(\gamma)}{e^{\nu \ell(\gamma)}}
\end{equation}
for every $\gamma \in \calP_{[L-\delta,L+\delta]}(M)$. Since both terms in the minimum are less than $1$, their product is smaller than either term and we find
\[
\frac{\delta}{e^{3\nu \delta}} \cdot  \frac{L}{e^{\nu L}} \leq \frac{2\delta}{e^{2\nu \delta}} \cdot \frac{(L+\delta)}{e^{\nu(L+\delta)}} < \min\left( \frac{2\delta}{e^{2\nu \delta}}, \frac{L+\delta}{e^{\nu(L+\delta)}} \right)
\]
Together with inequalities \eqref{eq:prim1} and  \eqref{eq:prim2}, this yields
\[
\frac{\mu \delta}{2^n e^{3\nu \delta}} \cdot  \frac{L}{e^{\nu L}} \cdot \# \calP_{[L-\delta,L+\delta]}(M) \leq  \sum_{\gamma\in \calC(M)}  \frac{\Lambda(\gamma)}{{N_\gamma}^{\nu} D(\gamma)}  G^L_\eps(\ell(\gamma)).
\]

The Selberg trace formula states that
\[
\sum_{\gamma\in \calC(M)}  \frac{\Lambda(\gamma)}{{N_\gamma}^{\nu} D(\gamma)} G^L_\eps(\ell(\gamma)) = \sum_{j=0}^\infty H^L_\eps(r_j) -  \vol(M) \int_0^\infty H^L_\eps(r) \Phi_n(r)\, dr
\]
and we now proceed to bound the right-hand side from above. Let $K$ be the maximum of $h_\eps$ on $i[-\nu,\nu]$. Then \[H^L_\eps(it) \leq K(\cosh(Lt) - 1) \leq K e^{Lt} \leq K e^{\nu L}\] for every $t \in [0, \nu]$. In particular, we have $H^L_\eps(r_j) \leq  K e^{\nu L}$ for every eigenvalue $\lambda_j$ in the interval $[0, \nu^2]$ (such eigenvalues are called \emph{small}). It is known that there exists a constant $W_n>0$ such that the number of small eigenvalues does not exceed $W_n \vol(M)$ for any closed hyperbolic $n$-manifold $M$. This is due to Buser for $n=2$ \cite{BuserSmall} (see \cite{OtalRosas} for the sharp version) and $n=3$ \cite{BuserDim3}, and to Buser--Colbois--Dodziuk for $n \geq 4$  \cite[Theorem 3.6]{BCD}.  Since $H^L_\eps$ is non-positive in $\RR$, we obtain
\[
\sum_{j=0}^\infty H^L_\eps(r_j) \leq \sum_{r_j \notin \RR} H^L_\eps(r_j) = \sum_{\lambda_j \text{ small}} H^L_\eps(r_j) \leq W_n \vol(M) K e^{\nu L}.
\] 
Recall that $H^L_\eps \geq -2h_0$ in $\RR$ so that
\[
-  \vol(M) \int_0^\infty H^L_\eps(r) \Phi_n(r)\, dr \leq \vol(M)  \int_0^\infty 2 h_\eps(r) \Phi_n(r)\, dr.
\]
If we denote the last integral by $I$, we have shown that
\[
\frac{\mu \delta}{2^n e^{3\nu \delta}} \cdot  \frac{L}{e^{\nu L}} \cdot \#\calP_{[L-\delta,L+\delta]}(M) \leq \vol(M) \left( KW_n e^{\nu L}+ I \right) \leq  \vol(M) \left( KW_n+ I \right) e^{\nu L}.
\]
Upon rearranging, we obtain
\[
\# \calP_{[L-\delta,L+\delta]}(M) \leq B_{n,\delta} \vol(M) \frac{e^{2\nu L}}{L}
\]
where $B_{n,\delta} :=2^n e^{3\nu \delta}(KW_n+I) / (\mu \delta)$ depends on $n$ and $\delta$ but not on $M$.
\end{proof}

This leads to the following upper bound on the number of primitive closed geodesics of bounded length.

\begin{cor}\label{cor:primgeodupperbound}
For every $n\geq 2$ and $\delta>0$, there exists a constant $B_{n,\delta}'>0$ such that for every closed hyperbolic $n$-manifold $M$ with $\sys(M)\geq 2\delta$ and every $L>0$, we have
\begin{equation} \label{eq:PNTcor}
\#\calP_{[0,L]}(M) \leq B_{n,\delta}'  \vol(M) \frac{e^{(n-1)L}}{L}.
\end{equation}
\end{cor}
\begin{proof}
We split the count into two parts. By subdividing the interval $[0,3]$ into $\lceil 3/(2\delta) \rceil$ subintervals of equal length $\leq 2\delta$ and applying \propref{prop:primupper} to each of these, we get some constant $F_{n,\delta}$ such that 
\[\#\calP_{[0,3]}(M) \leq F_{n,\delta} \vol(M).\]
Since $e^{(n-1)x} /x$ is bounded below by $(n-1)e$ for all $x>0$, inequality \eqref{eq:PNTcor} is true if $L \leq 3$.

If $L >3$, we also estimate
\[
\delta\cdot \#\calP_{[3,L]}(M) \leq \int_3^{L} \#\calP_{[x-\delta,x+\delta]}(M) \,dx \leq B_{n,\delta} \vol(M)  \int_3^{L} \frac{e^{(n-1)x}}{x} \,dx
\]
since any primitive closed geodesic $\gamma$ with length in the interval $[3,L]$ contributes to\linebreak $\#\calP_{[x-\delta,x+\delta]}(M)$ over an interval of length at least $\delta$.  The last inequality in the above comes from \propref{prop:primupper}.

After the change of variable $u = e^{(n-1)x}$, we get the logarithmic integral function
\[
 \int_3^{L} \frac{e^{(n-1)x}}{x} \,dx = \int_{\exp(3(n-1))}^{\exp((n-1)L)} \frac{du}{\log u} = \li\left( e^{(n-1)L}\right) - \li\left( e^{3(n-1)}\right) \leq \li\left( e^{(n-1)L}\right).
\]
Since $e^{3(n-1)} \geq 11$ for all $n\geq 2$ and 
\[ \li(x) \leq  \frac{x}{\log (x)} + \frac{x}{\log (x)^2} + \frac{3x}{\log (x)^3},  \]
for all $x\geq 11$, we find that $\#\calP_{[3,L]}(M)$ satisfies an inequality of the form \eqref{eq:PNTcor}. Adding the contribution of $\#\calP_{[3,L]}(M)$ cause no harm according to the first paragraph.
\end{proof}

\subsection{Lower bound}

The lower bound on the number of primitive closed geodesics with length in a small interval takes the following form:
\begin{prop}\label{prop:primlower}
For every $n\geq 2$ and $\delta>0$, there exist positive constants $C_{n,\delta}$ and $D_{n,\delta}$ such that for every closed hyperbolic $n$-manifold $M$ with $\sys(M) \geq \delta$ and every $L>0$, 
\[
\#\calP_{[L-\delta,L+\delta]}(M) \geq  C_{n,\delta} \frac{e^{(n-1)L}}{L} - D_{n,\delta} \vol(M) \frac{e^{\frac{(n-1)}{2} L}}{L}.
\]
\end{prop}
\begin{proof}
As before, set $\nu := (n-1)/2$. Let $(g_\delta,h_\delta)$ be the admissible transform pair given by \lemref{lem:convo} such that $g_\delta$ has support in $[-\delta,\delta]$. One can show that $g_\delta$ is non-increasing in $[0,\infty)$ by induction on the number of convolutions. In particular, it attains its maximum at the origin. In any case, what matters is that $g_\delta$ is bounded.

For $L>0$, consider the admissible transform pair given by 
\[G^L_\delta(x) = g_\delta(x) + \frac{g_\delta(x-L)+g_\delta(x+L)}{2} \quad \text{ and }\quad H^L_\delta(\xi) = (1+\cos(L\xi))h_\delta(\xi).\] 
Then $G^L_\delta$ is non-negative in $\RR$ and bounded above by $2g_\delta(0)$. Moreover, $H^L_\delta$ is non-negative in $\RR \cup i \left[-\nu, \nu\right]$ and bounded above by $2 h_\delta$ on the real line.

Set $I:= \int_0^\infty 2 h_\delta(r)\Phi_n(r)\,dr$ and let $M$ be any closed hyperbolic $n$-manifold such that $\sys(M) \geq \delta$. Recall that the root corresponding to the eigenvalue $\lambda_0=0$ is $r_0 = i \nu$. From the Selberg trace formula and the properties of $H^L_\delta$, we obtain 
\begin{align*}
(1+\cosh(L\nu)) h_\delta(i\nu)- \vol(M) I &= H^L_\delta(i\nu)- \vol(M) \int_0^\infty 2 h_\delta(r)\Phi_n(r)\,dr \\
&\leq  \sum_{j=0}^\infty H^L_\delta(r_j) -  \vol(M) \int_0^\infty H^L_\delta(r) \Phi_n(r)\, dr \\
&= \sum_{\gamma\in \calC(M)} \frac{\Lambda(\gamma)}{{N_\gamma}^{\nu}D(\gamma)}  G^L_\delta(\ell(\gamma)).
\end{align*}

The hypothesis that $\sys(M) \geq \delta$ implies that $G^L_\delta(\ell(\gamma)) = g_\delta(\ell(\gamma)-L)/2$ for every closed geodesic $\gamma$ in $M$ since $g_\delta$ vanishes outside $(-\delta,\delta)$. In particular, the summands in the above sum vanish unless $\ell(\gamma) \in [L-\delta, L + \delta]$. Furthermore, we have $G^L_\delta(\ell(\gamma)) \leq g_\delta(0)/2$ for every $\gamma \in \calC(M)$. We then use the inequality $(1-N_\gamma^{-1})^{2\nu} \leq D(\gamma)$ from \lemref{lem:det}, the trivial bound $\Lambda(\gamma)\leq \ell(\gamma)$ and the identity
\[
\left(1-N_\gamma^{-1}\right)^{2\nu} N_\gamma^{\nu} = \left(N_\gamma^{\frac12}-N_\gamma^{-\frac12}\right)^{2\nu} = \left(2 \sinh(\ell(\gamma)/2) \right)^{2\nu}
\]
to obtain
\[
\sum_{\gamma \in \calC(M)} \frac{\Lambda(\gamma)}{{N_\gamma}^{\nu}D(\gamma)} G^L_\delta(\ell(\gamma)) \leq \frac{g_\delta(0)}{2} \sum_{\gamma \in \calC_{[L-\delta,L+\delta]}(M)} \frac{\ell(\gamma)}{(2\sinh(\ell(\gamma)/2))^{2\nu}},
\]
where $\calC_J(M)$ is the set of closed oriented geodesics in $M$ with length in the interval $J$.

On the interval $[\delta, \infty)$, the function $2\sinh(x/2)e^{-x/2} = 1 - e^{-x}$ is bouded below by $1-e^{-\delta}$ so that 
\[
(1-e^{-\delta})^{2\nu} e^{\nu\ell(\gamma)} \leq (2\sinh(\ell(\gamma)/2))^{2\nu}
\]
for every $\gamma \in \calC(M)$. We therefore have
\[
 \frac{g_0(0)}{2} \sum_{\gamma \in \calC_{[L-\delta,L+\delta]}(M)} \frac{\ell(\gamma)}{(2\sinh(\ell(\gamma)/2))^{2\nu}} \leq  \frac{g_0(0)}{2(1-e^{-\eps})^{2\nu}} \sum_{\gamma \in \calC_{[L-\delta,L+\delta]}(M)} \frac{\ell(\gamma)}{e^{\nu\ell(\gamma)}}
\]

Assume that $L \geq \delta$. Then for every $x \in [L-\delta,L+\delta]$ we have
\[
x \leq 2L \leq 2 L e^{\nu(x-L+\delta)} \quad \text{and hence} \quad \frac{x}{e^{\nu x}} \leq 2e^{\nu\delta} \frac{L}{e^{\nu L}}
\]
which gives
\[
\frac{g_0(0)}{2(1-e^{-\delta})^{2\nu}} \sum_{\gamma \in \calC_{[L-\delta,L+\delta]}(M)} \frac{\ell(\gamma)}{e^{\nu\ell(\gamma)}} \leq \frac{g_0(0) e^{\nu\delta} L}{(1-e^{-\delta})^{2\nu} e^{\nu L}} \cdot  \# \calC_{[L-\delta,L+\delta]}(M) .
\]

All in all, we have shown that
\begin{equation} \label{eq:PNTlower}
\#\calC_{[L-\delta,L+\delta]}(M) \geq \frac{(1-e^{-\delta})^{2\nu}}{g_0(0) e^{\nu\delta}}\left[ \frac{h_0(i\nu)}{2} e^{\nu L} - \vol(M) I \right] \frac{e^{\nu L}}{L}.
\end{equation}

We next need to subtract the contribution from non-primitive curves. Since we assume $\sys(M)\geq \delta$, any primitive geodesic $\gamma_0$ has at most two positive powers $\gamma_0^m$ whose length lands in the interval $[L-\delta, L+\delta]$. Moreover, the non-primitive geodesics $\gamma$ whose length land in that interval must satisfy $\Lambda(\gamma)\leq (L+\delta)/2$. Thus the number of non-primitive geodesics such that $\ell(\gamma)\in [L-\delta, L+\delta]$ is at most 
\[
2\cdot\# \calP_{[0, (L+\delta)/2]}(M) \leq 4 B_{n,\delta}'  \vol(M) \frac{e^{(n-1)(L+\delta)/2}}{L+\delta} \leq 4 e^{\nu \delta} B_{n,\delta}'  \vol(M) \frac{e^{\nu L}}{L}
\]
by Corollary \ref{cor:primgeodupperbound}. Subtracting this from \eqref{eq:PNTlower} yields the lower bound for $\#\calP_{[L-\delta,L+\delta]}(M)$.
\end{proof}

Finally, this implies a lower bound on the number of primitive closed geodesics of bounded length.
\begin{cor} \label{cor:prim_lower}
For every $n\geq 2$ and $\delta>0$, there exist constants $C_{n,\delta}',D_{n,\delta}'>0$ such that for every closed hyperbolic $n$-manifold $M$ with $\sys(M)\geq 2\delta$ and every $L>0$,
 \[
 \#\calP_{[0,L]}(M) \geq C_{n,\delta}' \frac{e^{(n-1)L}}{L} - D_{n,\delta}' \vol(M) \frac{e^{\frac{(n-1)}{2} L}}{L}.
\]
\end{cor}

\begin{proof}
The proof is very similar to that of Corollary \ref{cor:primgeodupperbound}. Again, we split the count into two parts. Since there exists a uniform lower bound on the volume of a closed hyperbolic $n$-manifold for every dimension $n$, we can ignore the geodesics with length in $[0,6]$, at the cost of enlarging our constant $D_{n,\delta}'$.

Using \propref{prop:primlower}, we obtain
\begin{multline*}
2\delta \cdot \#\calP_{[6,L]}(M) \geq \int_6^{L} \#\calP_{[x-\delta,x+\delta]}(M) \,dx\\
 \geq C_{n,\delta} \int_6^{L} \frac{e^{(n-1)x}}{x} \,dx - D_{n,\delta} \vol(M)  \int_6^{L} \frac{e^{(n-1)x/2}}{x} \,dx 
\end{multline*}
because any primitive geodesic $\gamma$ with length in $[6,L]$ contributes to $\#\calP_{[x-\delta,x+\delta]}(M)$ over an interval of length at most $2\delta$.

We have
\[
 \int_6^{L} \frac{e^{(n-1)x}}{x} \,dx = \li\left( e^{(n-1)L}\right) - \li\left( e^{6(n-1)}\right) \]
 and
 \[ \int_6^{L} \frac{e^{(n-1)x/2}}{x} \,dx = \li\left( e^{(n-1)L/2}\right) - \li\left( e^{3(n-1)}\right) .
\]
Since 
\[ \frac{x}{\log (x)} + \frac{x}{\log (x)^2} \leq \li(x) \leq  \frac{x}{\log (x)} + \frac{x}{\log (x)^2} + \frac{3x}{\log (x)^3},  \]
for all $x\geq 11$, the result follows.
\end{proof}

\thmref{thm:unifpgt} combines Propositions \ref{prop:primupper} and \ref{prop:primlower} into a single statement, while \corref{cor:unifpgt} combines the Corollaries \ref{cor:primgeodupperbound} and \ref{cor:prim_lower}.

\bibliography{biblio}
\bibliographystyle{alpha}
\end{document}